\documentclass{amsart}

\usepackage{amssymb,amsmath,amsthm,latexsym,amscd,amsfonts}
\usepackage{graphicx,subcaption}
\usepackage{bbold}
\usepackage{ifpdf}
\usepackage{import}
\usepackage{color}
\ifpdf
  \usepackage{transparent}
\fi
\usepackage{subcaption}
\makeatletter

\protected\def\p@subfigure{\thefigure}
\makeatother
\usepackage{psfrag}
\usepackage{comment}
\usepackage[
  colorlinks=true,
  linkcolor=blue,
  citecolor=red,
  urlcolor=blue
]{hyperref}

\numberwithin{equation}{section}

\theoremstyle{plain}
\newtheorem{theorem}{Theorem}
\newtheorem{lemma}[theorem]{Lemma}
\newtheorem{proposition}[theorem]{Proposition}

\theoremstyle{definition}

\theoremstyle{remark}

\graphicspath{{Figures/}}

\newcommand{\be}{\begin{equation}}
\newcommand{\ee}{\end{equation}}

\newcommand{\LF}[1]{{LF}\left({#1}\right)}

\begin{document}

\title[Two-sided Guionnet-Jones-Shlyakhtenko construction]{On the two-sided Guionnet-Jones-Shlyakhtenko construction at level three}
\author{R Jayakumar}
\address{The Institute of Mathematical Sciences,
Chennai and
Homi Bhabha National Institute\\
Mumbai}
\date{\today}

\begin{abstract}
We study the two-sided Guionnet-Jones-Shlyakhtenko construction applied to the group planar algebra $P(\mathcal{G})$ of a finite non-trivial group $\mathcal{G}$. This produces a sequence of von Neumann algebras $M^k$ for $k \geq 0$ with no natural inclusions.
Focusing on level $k=3$, we show that the resulting von Neumann algebra $M^3$ is isomorphic to the interpolated free group factor $\LF{1+\frac{2(n-1)}{n^2}}$, where $n=|\mathcal{G}|$. 
Our approach keeps the combinatorics explicit and relies on standard tools from free probability and planar algebras.
\end{abstract}


\maketitle

\section{Introduction}

The two-sided Guionnet-Jones-Shlyakhtenko (GJS) construction - see \cite{JayakumarKodiyalam2025} - builds tracial von Neumann algebras $M^k$ for $k \geq 0$ from an irreducible subfactor planar algebra $P$ with modulus $\delta>1$. When $P=P(\mathcal{G})$ is the group planar algebra of a finite non-trivial group $\mathcal{G}$, the first two algebras $M^1$ and $M^2$ are already known to be interpolated free group factors as shown in  \cite{JayakumarKodiyalam2025}.

Our focus in this paper is $M^3$. We prove that $M^3$ is also an interpolated free group factor and give a clean, self-contained proof using planar algebra techniques combined with free probability methods.

The key insight is the identification of a specific element $Z \in P_{(5,+)}$ that  together with $P_{(3,+)}$ (which is just $M_n(\mathbb{C})$) generates the dense $*$-algebra $A_G^3$ of $M^3$. We will see that $Z$ corresponds to a free Poisson element, thereby deducing the structure of $A_G^3$.
Working inside the two-sided GJS construction with planar algebra diagrams makes the freeness between $Z$ and $P_{(3,+)}$ completely transparent, avoiding complications that occur while determining $M^1$ and $M^2$.
From this perspective, one immediately sees that $M^3$ coincides with the matrix-amplified algebra $M_n(M^1)$ where $n=|{\mathcal G}|$. 

The paper is organised as follows.
\begin{itemize}
  \item Section~\ref{sec:prelim} reviews the background on non-crossing partitions, Möbius inversion, free Poisson elements, and the structure of group planar algebras.
  \item Section~\ref{sec:gjs-setup} recalls the two-sided GJS construction, records the filtered algebra $A_F^k$ and the graded algebra $A_G^k$, and notes the trace-preserving isomorphism between them, which is essential for defining the resulting von Neumann algebras $M^k$.
  \item Section~\ref{sec:det-M3} contains the main result: we show that $Z$ and $P_{(3,+)}$ generate $A_G^3$, prove that they are free, and conclude that $M^3 \cong \LF{1+\frac{2(n-1)}{n^2}}$.
\end{itemize}

\section{Preliminaries}\label{sec:prelim}

We collect the free probability background and the group planar algebra facts used in Section~\ref{sec:det-M3}. For the convenience of the reader, we briefly recall the key concepts, assuming familiarity with basic planar algebra and free probability theory. Our exposition follows the methodology and techniques of \cite{JayakumarKodiyalam2025}. For free probability theory we rely on \cite{NicaSpeicher2006} and \cite{VoiculescuDykemaNica1992}, while \cite{Jones2021} and \cite{De2016} cover the planar algebra background.

\subsection{Tracial non-commutative probability spaces}

A pair $(A,\varphi)$ consisting of a unital algebra $A$ and a linear functional $\varphi\colon A\to\mathbb{C}$ with $\varphi(1)=1$ is called an algebraic non-commutative probability space (NCPS). If $A$ is a unital $*$-algebra and $\varphi$ is positive, that is, $\varphi(a^*a)\geq 0$ for all $a\in A$, we obtain a $*$-NCPS; if, moreover, $A$ is a von Neumann algebra and $\varphi$ is a faithful normal state, we speak of a von Neumann NCPS.

Throughout this paper we work with tracial $*$-NCPS, namely $\varphi(ab)=\varphi(ba)$ for all $a,b\in A$. The elements of $A$ are referred to as non-commutative random variables.

A family of unital $*$-subalgebras $\{A_i : i\in I\}$ is \emph{free} if $\varphi(a_1\cdots a_q)=0$ whenever $q\geq 1$, $a_j\in A_{i_j}$, $\varphi(a_j)=0$, and adjacent indices satisfy $i_j\neq i_{j+1}$. 
A family of random variables $\{x_i : i\in I\}\subset A$ is free (resp.~$*$-free) when the unital (resp.~unital $*$-)subalgebras they generate form a free family. 
Freeness is preserved by taking von Neumann closures: if $(M,\tau)$ is a tracial von Neumann algebra and $\{B_i\}$ is a family of unital $*$-subalgebras, then $\{B_i\}$ is free if and only if $\{B_i''\}$ is free \cite[Lemma~2.5.7]{NicaSpeicher2006}.

\subsection{Free products.}\;Given a family of tracial NCPS's $(A_i,\varphi_i)_{i\in I}$ their \emph{free product}
\[
  (A,\varphi)=*_{i\in I}(A_i,\varphi_i)
\]
is the unique tracial NCPS characterised by the following two conditions:
\begin{enumerate}
  \item $A$ is the algebraic free product of the unital algebras $A_i$ amalgamated over the common unit;
  \item the state $\varphi$ restricts to $\varphi_i$ on each $A_i$ and satisfies
  \(
      \varphi(a_1\cdots a_q)=0
  \)
  whenever $q\ge 1$, $a_j\in A_{i_j}^{\circ}:=\{a\in A_{i_j}:\varphi_{i_j}(a)=0\}$ and adjacent indices differ: $i_j\ne i_{j+1}$.
\end{enumerate}
Equivalently, the inclusion maps $A_i\hookrightarrow A$ are free with respect to $\varphi$ and collectively generate $A$.  In the von~Neumann setting the same construction, followed by completion in the GNS representation of $\varphi$, yields the \emph{free product von~Neumann algebra} $*_{i\in I}(M_i,\tau_i)$, characterised by the same universal freeness property.

\subsection{Non-crossing partitions and free cumulants}

Let $[q]=\{1,\ldots,q\}$. A partition $\pi$ of $[q]$ is \emph{non-crossing} if its blocks can be drawn without crossings on a circle with $q$ labelled points on the boundary.
The lattice of non-crossing partitions on $[q]$ is denoted $NC(q)$; we write $\rho\leq\pi$ when $\rho$ refines $\pi$, and denote by $0_q$ and $1_q$ its minimal and maximal elements.
While we restrict to $[q]$ in this section, the same discussion works for any finite totally ordered set; see Nica and Speicher~\cite{NicaSpeicher2006} for a detailed treatment.

Let $S$ be any set and let $\{\phi_q : S^q \to \mathbb{C}\}_{q\in\mathbb{N}}$ be functions. Their multiplicative extension is the family
\[
\{\phi_\pi : S^q \to \mathbb{C}\}_{q\in\mathbb{N},\,\pi\in NC(q)},
\]
defined by
\[
\phi_\pi(a_1,\ldots,a_q)=\prod_{C\in\pi}\phi_{|C|}(a_c : c\in C),
\]
where, inside each block $C$, the entries are taken in increasing order. In particular $\phi_q=\phi_{1_q}$. When the maps $\phi_q$ are the moments $\varphi_q$ of a NCPS $(A,\varphi)$ we write $\varphi_\pi$, and applying the same rule to the cumulants produces $\kappa_\pi$.

The \emph{free cumulants} $\{\kappa_q\}_{q\geq 1}$ of $(A,\varphi)$ are determined by the relations
\begin{align*}
\varphi(a_1\cdots a_q)=\varphi_q(a_1,\ldots,a_q)&=\sum_{\pi\in NC(q)}\kappa_\pi(a_1,\ldots,a_q),\\
\kappa_q(a_1,\ldots,a_q)&=\sum_{\pi\in NC(q)}\mu(\pi,1_q)\,\varphi_\pi(a_1,\ldots,a_q),\label{eq:mobius}
\end{align*}
where $\mu$ is the Möbius function on $NC(q)$.

The next result, often called Möbius inversion, collects equivalent formulations that we use in Section~\ref{sec:det-M3}; see Kodiyalam and Sunder~\cite[Theorem~7]{KodiyalamSunder2009}.

\begin{theorem}\label{thm:multiplicative-cumulants}
Given two collections of functions $\{\phi_n : S^n \to \mathbb{C}\}_{n\in\mathbb{N}}$ and $\{\kappa_n : S^n \to \mathbb{C}\}_{n\in\mathbb{N}}$ extended multiplicatively, the following conditions are all equivalent:
\begin{enumerate}
  \item $\phi_n = \sum_{\pi\in NC(n)} \kappa_\pi$ for each $n\in\mathbb{N}$.
  \item $\kappa_n = \sum_{\pi\in NC(n)} \mu(\pi,1_n)\,\phi_\pi$ for each $n\in\mathbb{N}$.
  \item $\phi_\tau = \sum_{\pi\in NC(n),\,\pi\leq\tau} \kappa_\pi$ for each $n\in\mathbb{N}$ and $\tau\in NC(n)$.
  \item $\kappa_\tau = \sum_{\pi\in NC(n),\,\pi\leq\tau} \mu(\pi,\tau)\,\phi_\pi$ for each $n\in\mathbb{N}$ and $\tau\in NC(n)$.
\end{enumerate}
\end{theorem}

The following criterion, due to Speicher~\cite[Theorem~11.20]{NicaSpeicher2006}, says that it suffices to test freeness on chosen generators.

\begin{theorem}\label{thm:Speicher-cumulants-detect-freeness}
Let $\{A_i : i\in I\}$ be unital subalgebras generated by sets $S_i$ inside a tracial $*$-NCPS. Then $\{A_i\}$ is free if and only if
\[
\kappa_q(a_1,\ldots,a_q)=0
\]
for every $q\geq 1$, whenever $a_j\in S_{i_j}$ and not all indices $i_1,\ldots,i_q$ are equal.
\end{theorem}

For later use we state a simple lemma which is proved in \cite[Lemma~6.4]{JayakumarKodiyalam2025}.

\begin{lemma}\label{lem:pi-tilde}
Let $q\geq 1$ and decompose $[q]=D\sqcup E$. Given $\pi\in NC(D)$, there exists a unique $\widetilde{\pi}\in NC(E)$ such that
\begin{enumerate}
  \item $\pi\sqcup\widetilde{\pi}$ is a non-crossing partition of $[q]$;
  \item whenever $\rho\in NC(E)$ satisfies $\pi\sqcup \rho\in NC(q)$, we have $\rho\leq \widetilde{\pi}$.
\end{enumerate}
\end{lemma}
Informally, $\widetilde{\pi}$ is the largest non-crossing partition on $E$ such that $\pi\sqcup\widetilde{\pi}$ remains non-crossing.

\subsection{Weighted direct sums and Free Poisson elements}
For a countable discrete group $\mathcal{G}$ we write $L\mathcal{G}$ for its group von Neumann algebra with canonical trace $\tau(\lambda(g))=\delta_{g,e}$. If $F_n$ is the free group on $n$ generators, then $LF({n})$ denotes the corresponding free group factor; in particular $\LF{1}=L\mathbb{Z}$ and $\LF{n}$ is a type~$\mathrm{II}_1$ factor for every $n\geq 2$.

Given tracial NCPS's $(A,\varphi_A)$ and $(B,\varphi_B)$, and a scalar $\alpha \in (0,1)$, the notation $\underset{\alpha}{A}\oplus\underset{1-\alpha}{B}$ indicates the direct sum algebra equipped with the convex combination of states $\alpha\,\varphi_A + (1-\alpha)\,\varphi_B$.

A self-adjoint element $x$ in a tracial $*$-NCPS $(A,\varphi)$ is a \emph{free Poisson element} with rate $\lambda\geq 0$ and jump size $\alpha\in\mathbb{R}$ when its free cumulants satisfy $\kappa_q(x,\ldots,x)=\lambda\alpha^q$ for every $q\geq 1$. If $B={vN}(x)$ denotes the von Neumann algebra generated by $x$, then $(B,\varphi|_B)$ is canonically isomorphic to the weighted direct sum $\underset{1-\lambda}{\mathbb{C}}\oplus\underset{\lambda}{L\mathbb{Z}}$ \cite[Proposition~13]{KodiyalamSunder2009}.

\subsection{Interpolated free group factors}

Interpolated free group factors - found independently by Dykema \cite{Dykema1994} and R\u{a}dulescu \cite{Radulescu1995}, building on Voiculescu's free probability framework - extend the family $\{\LF{n} : n\in\mathbb{N}\}$ to $\{\LF{r} : r\in [1,\infty]\}$.
In the sequel we will only need the following free product formula from \cite[Proposition~1.4.9]{Dykema1994}.
\begin{equation*}\label{eq:LF-matrix}
\Bigl( \underset{1-\alpha}{\mathbb{C}} \oplus \underset{\alpha}{\LF{r}} \Bigr) * M_d(\mathbb{C})
\;=\;
\begin{cases}
\displaystyle \LF{\,r\alpha^{2}+2\alpha(1-\alpha)+1-d^{-2}\,} & \quad (\alpha \ge d^{-2}),\\[6pt]
\displaystyle \underset{1-\alpha d^{2}}{M_d(\mathbb{C})}\;\oplus\;\underset{\alpha d^{2}}{\LF{(r-2)d^{-4}+1+d^{-2}}} & \quad (\alpha \le d^{-2}).
\end{cases}
\end{equation*}

\subsection{Group planar algebras}

Let $\mathcal{G}$ be a finite group  and let $P(\mathcal{G})$ denote the associated group planar algebra in the sense of Jones \cite{JonesShlyakhtenkoWalker2010}. This is a connected, irreducible subfactor planar algebra with modulus $\delta = \sqrt{n}$ and $*$-structure coming from inversion in $\mathcal{G}$; see \cite{JonesShlyakhtenkoWalker2010} for background. The following result of Landau \cite[Theorem~2.4.2]{Landau2002} will be used repeatedly.

\begin{theorem}\label{thm:landau-basis}
For any $(k,\epsilon)$, if $S$ is a $(k,\epsilon)$-tangle with no internal region and exactly $k-1$ internal boxes, then the diagrams obtained by labelling those boxes by elements of $\mathcal{G}$ form a basis of $P_{(k,\epsilon)}(\mathcal{G})$. In particular $\dim P_{(k,\epsilon)}(\mathcal{G}) = n^{k-1}$.
\end{theorem}

\section{The two-sided GJS construction}\label{sec:gjs-setup}

Let $P$ be an irreducible subfactor planar algebra of modulus $\delta>1$.
We will describe the two-sided GJS construction of \cite{JayakumarKodiyalam2025} whose take-off point is the original GJS construction of \cite{GuionnetJonesShlyakhtenko2010}.
Both of these associate to $P$, a sequence of tracial $*$-algebras and, after completion, a sequence of tracial von Neumann algebras.

For every integer $k\geq 0$ set
\[
A^k = \bigoplus_{m\geq 0} P_{(2m+k,+)}.
\]
The summand $A_m^k = P_{(2m+k,+)}$ is pictured as a $(2m+k,+)$-box, as illustrated in Figure~\ref{fig:typical-Ak} and matching the convention of \cite{JayakumarKodiyalam2025}.

\begin{figure}[ht]
  \centering
  \includegraphics[width=0.15\textwidth]{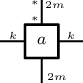}
  \caption{A typical element of $A_m^k$.}
  \label{fig:typical-Ak}
\end{figure}

 Two products live on $A^k$.
The \emph{filtered} product $\#$ of $a_m\in A_m^k$ and $b_n\in A_n^k$ is a sum of components $(a\#b)_t \in A_t^k$ 
with $t$ varying from $|m-n|$ to $m+n$ where $(a\#b)_t$ is as in Figure \ref{fig:filtered-product}.
The \emph{graded} product $\cdot$ multiplies $a_m$ and $b_n$ by simple concatenation, producing an element of $A_{m+n}^k$ as in Figure~\ref{fig:graded-product}.

\begin{figure}[ht]
  \centering
  \begin{subfigure}[b]{0.41\textwidth}
    \centering
    \includegraphics[width=\textwidth]{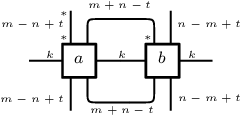}
    \caption{Filtered multiplication $(a \# b)_t$}
    \label{fig:filtered-product}
  \end{subfigure}
  \hfill
  \begin{subfigure}[b]{0.34\textwidth}
    \centering
    \includegraphics[width=\textwidth]{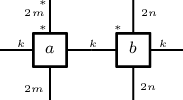}
    \caption{Graded multiplication $a\cdot b$}
    \label{fig:graded-product}
  \end{subfigure}
  \caption{Two multiplication operations on $A^k$}
\end{figure}

The conjugate-linear involution $\dagger$ is defined component-wise: for $a_m\in A_m^k$ we set $a_m^{\dagger}$ to be the planar algebra adjoint $a_m^{*}$ followed by the rotation tangle that restores the distinguished arc of external box back to the original position; see Figure~\ref{fig:dagger}. Extending additively over $A^k$ yields $\dagger$, and with this choice both $(A^k,\#,\dagger)$ and $(A^k,\cdot,\dagger)$ become $*$-algebras, denoted $A_F^k$ and $A_G^k$, respectively \cite[Section~3]{JayakumarKodiyalam2025}.

\begin{figure}[ht]
  \centering
   \includegraphics[width=0.5\textwidth]{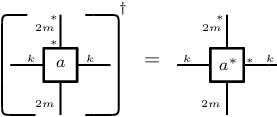}
  \caption{Obtaining $a^{\dagger}$ from $a$: apply the planar algebra adjoint and then the rotation tangle.}
  \label{fig:dagger}
\end{figure}

Two tracial functionals are defined diagrammatically: $\tau_F$ on $A_F^k$ and $\tau_G$ on $A_G^k$. Explicitly, if $a=\sum_{m\ge 0} a_m\in A_F^k$ with $a_m\in P_{(2m+k,+)}$, we set
\[
  \tau_F(a)=\tau(a_0),
\]
where $\tau$ is the normalized pictorial trace on the $(k,+)$-box space $P_{(k,+)}$.  Thus $\tau_F$ simply extracts the degree--$0$ component and applies the planar–algebra trace.
Figure~\ref{fig:traceG} displays the planar tangle for $\tau_G$. Here, ${TL}(m)$ denotes the set of all Temperley–Lieb tangles with $2m$ boundary points. The two traces agree on $P_{(k,+)}$ and satisfy $\tau_F\circ\Phi = \tau_G$, where
\[
\Phi\colon A_G^k \longrightarrow A_F^k, \qquad \Psi\colon A_F^k \longrightarrow A_G^k
\]
are mutually inverse trace-preserving $*$-isomorphisms built by summing over $k$-good and $k$-excellent annular Temperley-Lieb tangles. For brevity, we leave the details and diagrammatic representation of such tangles to \cite[Section~4]{JayakumarKodiyalam2025}. In particular $\tau_G$ is a faithful trace.

\begin{figure}[ht]
  \centering
  \includegraphics[width=0.40\textwidth]{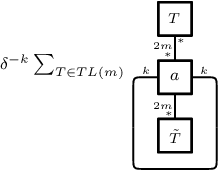}
  \caption{Diagrammatic formula for the trace $\tau_G$ on $A_G^k$.}
  \label{fig:traceG}
\end{figure}

Let $H^k$ be the GNS Hilbert space obtained by completing $A_F^k$ with respect to $\tau_F$, and let
\[
\lambda_k\colon A_F^k \to B(H^k),\qquad \lambda_k(a)(b) = a \# b,
\]
be the left regular representation. Pictorial estimates from \cite[Section~5]{JayakumarKodiyalam2025} show that $\lambda_k(a)$ is bounded for every $a\in A_F^k$. Consequently,
\begin{equation}\label{eq:GJS-tower}
M^k = \lambda_k(A_F^k)^{\prime\prime} = \lambda_k(A_G^k)^{\prime\prime}
\end{equation}
is a tracial von Neumann algebra whose faithful normal trace extends $\tau_F$. The algebras $\{M^k : k\geq 0\}$ form the two-sided GJS sequence attached to $P$; by construction $A_G^k$ is a weakly dense $*$-subalgebra of $M^k$.

In the remainder of the paper we specialise to the group planar algebra $P(\mathcal{G})$. The two-sided construction for this planar algebra was initiated in \cite{JayakumarKodiyalam2025}, where the algebras $M^1$ and $M^2$ were identified. Our goal in Section~\ref{sec:det-M3} is to continue this programme by determining $M^3$.

\section{Determining \texorpdfstring{$M^3$}{M3}}\label{sec:det-M3}

Throughout this section we fix a non-trivial finite group $\mathcal{G}$ of order $n=|\mathcal{G}|$. Let $P=P(\mathcal{G})$ be the associated group planar algebra with modulus $\delta=\sqrt{n}$ and denote by $M^3$ the von Neumann algebra produced by the two-sided GJS construction at level $k=3$. Write $\tau_G$ for the canonical trace on the dense $*$-subalgebra $A_G^3\subset M^3$. Our analysis proceeds in three steps: we isolate a self-adjoint element $Z\in A_G^3$, prove that $Z$ and $P_{(3,+)}$ freely generate $A_G^3$, and then determine the von Neumann algebras ${vN}(Z)$ and $P_{(3,+)}$ separately. The final identification of $M^3$ hinges on the free product formula from \S 2.5. The main result of this note is the following theorem.

\begin{theorem}\label{thm:main-M3}
Let $M^3$ be as above. Then
\[
M^3 \cong LF\left(1+\frac{2(n-1)}{n^2}\right),
\]
so $M^3$ is a $II_1$ factor.
\end{theorem}

\subsection{A dense \texorpdfstring{$*$}{*}-algebra and a distinguished generator}

Consider the algebra $A_G^3$ with its graded structure given by
\[
A_G^3 = \bigoplus_{m\geq 0} P_{(2m+3,+)}, \qquad \dim P_{(2m+3,+)} = n^{2m+2}.
\]
 Let $Z\in (A_G^3)_1 = A^3_1 = P_{(5,+)}$ be the self-adjoint element depicted in Figure~\ref{fig:element-Z}. 

\begin{figure}[ht]
  \centering
  \includegraphics[width=0.18\textwidth]{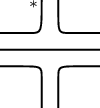}
  \caption{The self-adjoint element $Z\in P_{(5,+)}$.}
  \label{fig:element-Z}
\end{figure}

Our first goal is to show that $Z$ together with $P_{(3,+)}$  generates $A_G^3$.

\begin{proposition}\label{prop:AG3-gen}
The algebra $A_G^3$ is generated by $Z$ and $P_{(3,+)}$.
\end{proposition}

\begin{proof}
Let $\mathcal{B}$ be the algebra generated by $Z$ and $P_{(3,+)}$. We will show that $P_{(2m+3,+)} \subseteq \mathcal{B}$ for all $m \geq 0$.

Consider the tangle structure depicted in Figure~\ref{fig:gentang_2}. This figure illustrates a recursive construction where each stage $m$ produces a tangle with no internal regions and exactly $2m+2$ unlabeled $(2,+)$-boxes. By Theorem~\ref{thm:landau-basis}, labeling these boxes with elements of $\mathcal{G}$ yields $n^{2m+2}$ diagrams that form a linear basis for $P_{(2m+3,+)}$.

Note that each such tangle is constructed entirely from copies of the element $Z$ and elements from the subalgebra $P_{(3,+)}$ both of which are contained in $\mathcal{B}$.
Hence $P_{(2m+3,+)} \subseteq \mathcal{B}$ for each $m \geq 0$.
As $A_G^3 = \bigoplus_{m \geq 0} P_{(2m+3,+)}$, it follows that $A_G^3 = \mathcal{B}$.
\end{proof}

\begin{figure}[!ht]
  \centering
  \includegraphics[width=\textwidth]{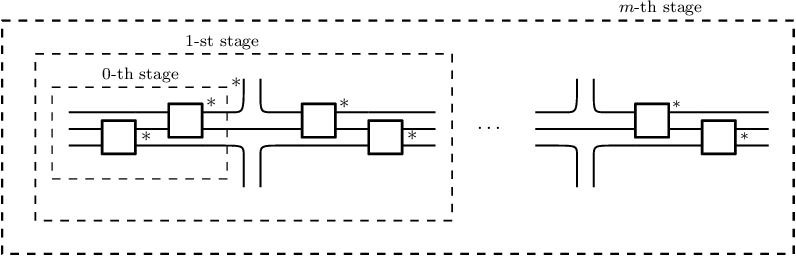}
  \caption{A generating tangle for $P_{(2m+3,+)}$.}
  \label{fig:gentang_2}
\end{figure}

Let $\mathcal{C}=\langle Z\rangle$ denote the $*$-algebra generated by $Z$. To compare $Z$ with earlier constructions, recall the family of elements $Y_a\in A_G^2$ described in \cite{JayakumarKodiyalam2025} and pictured in Figure~\ref{fig:Ya_Ye}. The element $Y_e$ corresponds to the identity of $\mathcal{G}$. 

\begin{figure}[!ht]
  \centering
  \includegraphics[width=0.45\textwidth]{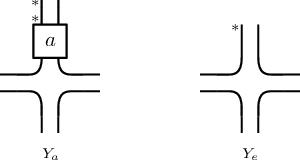}
  \caption{The elements $Y_a\in A_G^2$ for $a\in \mathcal{G}$.}
  \label{fig:Ya_Ye}
\end{figure}

\begin{proposition}\label{prop:YeZiso}
The $*$-algebras $\langle Y_e\rangle$ and $\mathcal{C}$ are $*$-isomorphic by a trace-preserving $*$-isomorphism that takes 
$Y_e$ to $Z$.
\end{proposition}

\begin{proof} It is clear that the map that takes $Y_e$ to $Z$ extends to give a $*$-isomorphism between $\langle Y_e\rangle$ and $\mathcal{C}$. The only thing to check is that this is trace preserving. In other words we need to see that $\tau_G(Y_e^m) = \tau_G(Z^m)$ for all $m$. From the definition of the two different $\tau_G$'s (one in $A^2_G$ and the other in $A^3_G$), each is a sum over Temperley-Lieb tangles $T \in TL(m)$. It suffices to check that the corresponding terms are the same. This follows from Figure \ref{fig:momenteqv}.
\end{proof}

\begin{figure}[!ht]
  \centering
  \includegraphics[width=\textwidth]{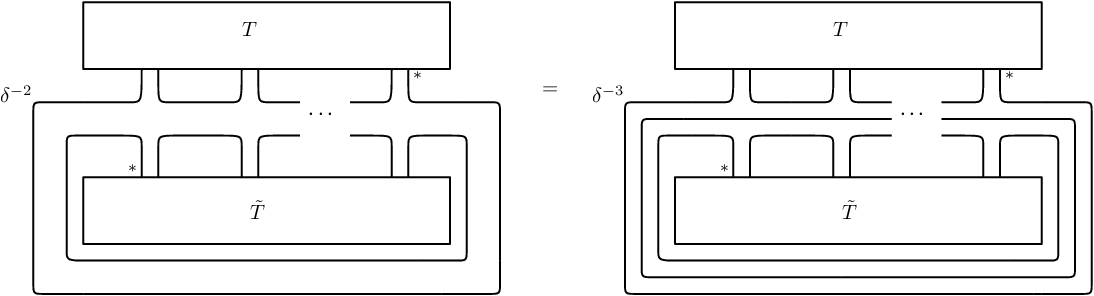}
  \caption{Planar tangles relating $Y_e$ and $Z$.}
  \label{fig:momenteqv}
\end{figure}

The identification of ${\mathcal C}$ with $\langle Y_e\rangle$ allows us to import the analytic information established in \cite{JayakumarKodiyalam2025}.
We record the consequence in the next lemma.

\begin{lemma}\label{lem:Z-free-poisson}
Let $Z\in A_G^3$ be the element from Figure~\ref{fig:element-Z}. 
Then, $Z$ satisfies:
\begin{enumerate}
  \item The free cumulants of $Z$ are $\kappa_q(Z,\ldots,Z)=\delta^{2q-2}$ for every $q\geq 1$.
  \item $Z$ is a free Poisson element with rate $\delta^{-2}$ and jump size $\delta^2$.
  \item The von Neumann algebra ${vN}(Z)$ is isomorphic to $\underset{1-\delta^{-2}}{\mathbb{C}}\oplus \underset{\delta^{-2}}{L\mathbb{Z}} = \underset{1-\frac{1}{n}}{\mathbb{C}}\oplus \underset{\frac{1}{n}}{L\mathbb{Z}}$.
\end{enumerate}
\end{lemma}

\begin{proof}
Each statement holds for $Y_e$ 
by \cite[Corollary~6.3]{JayakumarKodiyalam2025}
and \cite[Proposition 13]{KodiyalamSunder2009}. 
Because $\phi$ in Proposition~\ref{prop:YeZiso} is a trace-preserving $*$-isomorphism, it carries moments and cumulants of $Y_e$ to those of $Z$. The desired statements are then clear.
\end{proof}

\subsection{Freeness and an algebraic free product decomposition}

We now establish the freeness required to identify $M^3$.

\begin{proposition}\label{prop:C-and-P3-free}
The subalgebras $\mathcal{C}$ and $P_{(3,+)}$ are free in $(A_G^3,\tau_G)$.
\end{proposition}

\begin{proof}
Fix $q\geq 1$ and set $S=\{Z\}\cup P_{(3,+)}$. Choose elements $Z_1,\ldots,Z_q\in S$. Our goal is to prove that the mixed free cumulant $\kappa_q(Z_1,\ldots,Z_q)$ vanishes whenever the tuple contains at least one occurrence of $Z$ and at least one element of $P_{(3,+)}$.

\medskip\noindent\textbf{Step 1: Separating the indices.}
Let
\[
D=\{i\in[q] : Z_i = Z\}, \qquad E=[q]\setminus D.
\]
If either set is empty then every $Z_i$ belongs to the same subalgebra and there is nothing to prove, so we henceforth assume that both $D$ and $E$ are non-empty. For each $\pi\in NC(D)$, let  $\widetilde{\pi}\in NC(E)$ be as in Lemma~\ref{lem:pi-tilde}. As an example, suppose that $q=18$ and $D = \{2,5,8,11,13,14,17\}$ so that $E = \{1,3,4,6,7,9,10,12,15,16,18\}$ as shown in Figure \ref{fig:pi_term_example}.

\begin{figure}[!ht]
\centering
\includegraphics[width=1\textwidth]{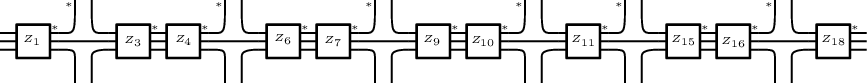}
\caption{Diagrammatic expansion of $Z_1Z_2\cdots Z_{18}$ illustrating the index sets $D$ and $E$.}
\label{fig:pi_term_example}
\end{figure}

\medskip\noindent\textbf{Step 2: Expanding the mixed moment.}
We write $\tau_G(Z_1\cdots Z_q)$ as a sum over $\pi\in NC(D)$ and consider the $\pi$-term.
To illustrate the construction explicitly, we choose $\pi = \{\{2, 8, 11\}, \{5\}, \{13, 14, 17\}\}$, which automatically determines $\widetilde{\pi} = \{\{1, 12, 18\}, \{3, 4, 6, 7\}, \{9, 10\}, \{15, 16\}\}$ - see Figure \ref{fig:pi_term}. Using sphericality this can be redrawn as Figure \ref{fig:pi_term_redrawn}.

\begin{figure}[!ht]
\centering
\includegraphics[width=\textwidth]{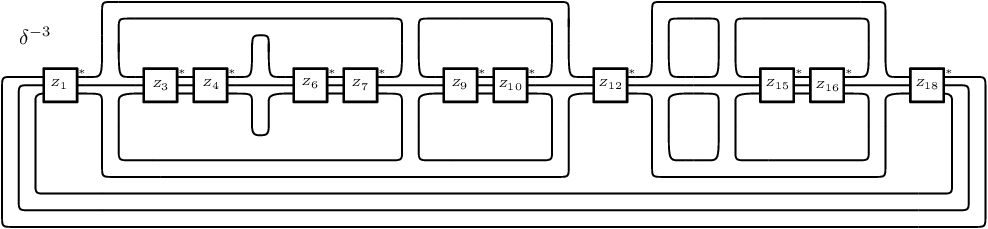}
\caption{The $\pi$-term}
\label{fig:pi_term}
\end{figure}

\begin{figure}[!ht]
\centering
\includegraphics[width=\textwidth]{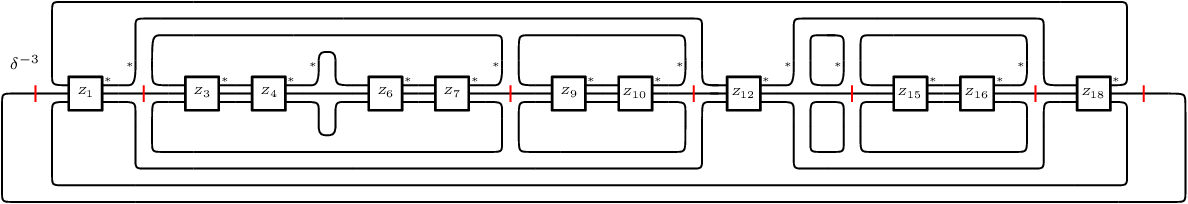}
\caption{The $\pi$-term redrawn}
\label{fig:pi_term_redrawn}
\end{figure}

To evaluate this summand, we perform the reductions indicated in Figure~\ref{fig:redrwn_pi_term_2}: floating loops are removed (in this example there are 2 of them that contribute a $\delta^2$) and then the rest of the picture breaks up into factors according to the classes of $\tilde{\pi}$, by irreducibility.  Each class, say $V$, of $\tilde{\pi}$ contributes a factor of $\delta^{-1}$ together with $\delta^3(\tau_G)_{|V|}(Z_i:i \in V)$, as shown in Figure~\ref{fig:redrwn_pi_term_2}. Besides this, there is a multiplicative factor of $\delta^{-3}$ and a single leftover loop that gives a $\delta$.

\begin{figure}[!ht]
\centering
\includegraphics[width=\textwidth]{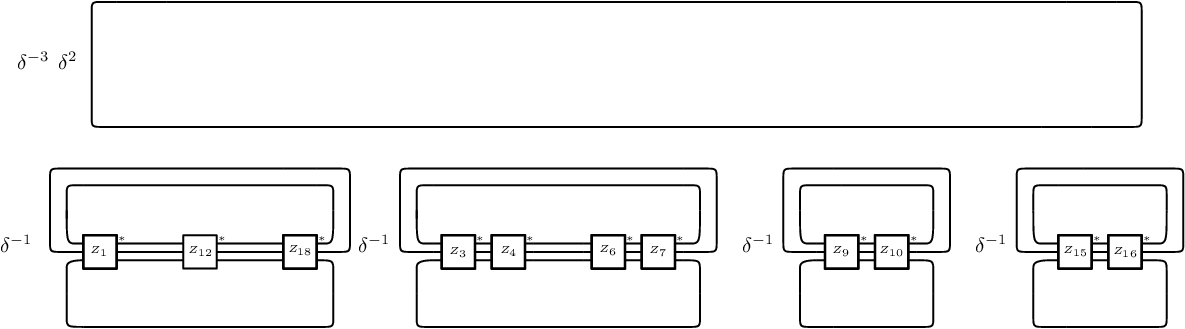}
\caption{Using irreducibility.}
\label{fig:redrwn_pi_term_2}
\end{figure}

Therfore the $\pi$-summand evaluates to
$$
\delta^{-3}.\delta.\delta^{\text {number of floating loops~}}.\delta^{2|\tilde{\pi}|}.(\tau_G)_{\widetilde{\pi}}\bigl(Z_j : j\in E\bigr),
$$
which by Lemma
\ref{lem:floating-loop} evaluates to
$$
\delta^{-2}.\delta^{2(|D|-|\pi|-|\tilde{\pi}|+1)}.\delta^{2|\tilde{\pi}|}.(\tau_G)_{\widetilde{\pi}}\bigl(Z_j : j\in E\bigr),
$$
which equals
$$
\delta^{2(|D|-|\pi|)}.(\tau_G)_{\widetilde{\pi}}\bigl(Z_j : j\in E\bigr).
$$
We now assert that 
$$
\delta^{2(|D|-|\pi|)} = \kappa_{\pi}(Z_i: i \in D).
$$
Since each $Z_i=Z$ for $i \in D$, this assertion is an easy consequence of Lemma \ref{lem:Z-free-poisson}(1).
Thus we have seen that
\begin{equation}\label{eq:factorization}
\tau_G(Z_1\cdots Z_q) = \sum_{\pi\in NC(D)} \kappa_\pi\bigl(Z_i : i\in D\bigr)\,(\tau_G)_{\widetilde{\pi}}\bigl(Z_j : j\in E\bigr),
\end{equation}.

In our particular example, consider the $\pi$-term of $\tau_G(Z_1Z_2 \ldots Z_{18})$ for our chosen partition $\pi$. It evaluates to
\begin{align*}
 & \delta^{-3}  \cdot\delta \cdot \delta^{2} \cdot (\delta^{2} \tau_G(Z_1 Z_{12} Z_{18})) \cdot(\delta^2\tau_G(Z_3 Z_4 Z_6 Z_7)) \cdot(\delta^2\tau_G(Z_9 Z_{10}))\cdot (\delta^2\tau_G(Z_{15} Z_{16}))\\
&= \delta^{-2} \cdot \delta^{2} \cdot \delta^{2(4)} \cdot  (\tau_G)_{\widetilde{\pi}}(Z_1, Z_3, Z_4, Z_6, Z_7, Z_9, Z_{10},Z_{12}, Z_{15}, Z_{16}, Z_{18})\\
&= \delta^{2(7-3)} \cdot  (\tau_G)_{\widetilde{\pi}}(Z_1, Z_3, Z_4, Z_6, Z_7, Z_9, Z_{10},Z_{12}, Z_{15}, Z_{16}, Z_{18})\\
&= \kappa_\pi(Z_2, Z_5, Z_8, Z_{11}, Z_{13}, Z_{14}, Z_{17}) \cdot (\tau_G)_{\widetilde{\pi}}(Z_1, Z_3, Z_4, Z_6, Z_7, Z_9, Z_{10},Z_{12}, Z_{15}, Z_{16}, Z_{18})
\end{align*}

\medskip\noindent\textbf{Step 3: Re-indexing by partitions of $[q]$.}
We now rewrite the right-hand side of Equation~\eqref{eq:factorization} as a sum over partitions of the whole set $[q]$. Recall that $S=\{Z\}\cup P_{(3,+)}$.

For any $n \in {\mathbb N}$ define $\tilde{\kappa}_n: S^n \rightarrow {\mathbb C}$ by
\[
\widetilde{\kappa}_n(Z_1,\ldots,Z_n) = \begin{cases}
\kappa_n(Z_1,\ldots,Z_n) & \text{if all $Z_i=Z$ or all $Z_i$ lie in $P_{(3,+)}$}, \\
0 & \text{otherwise.}
\end{cases}
\]

Then it is easy to see that the multiplicative extension of $\tilde{\kappa}_n$ satisfies
\[
\widetilde{\kappa}_\sigma(Z_1,\ldots,Z_q) = \begin{cases}
\kappa_\sigma(Z_1,\ldots,Z_q) & \text{if every block of $\sigma$ is contained } \\
& \text{either in } D \text{ or in } E, \\
0 & \text{otherwise.}
\end{cases}
\]

Therefore,
\[
\sum_{\sigma\in NC(q)} \widetilde{\kappa}_\sigma(Z_1,\ldots,Z_q) = \sum_{\pi\in NC(D)} \kappa_\pi(\{Z_i : i\in D\}) \left( \sum_{\rho\in NC(E),\,\rho\leq \widetilde{\pi}} \kappa_\rho(\{Z_j : j\in E\}) \right).
\]

From here it follows using Theorem ~\ref{thm:multiplicative-cumulants} that
\begin{equation*}\label{eq:moment-by-sigma}
\tau_G(Z_1\cdots Z_q) = \sum_{\sigma\in NC(q)} \widetilde{\kappa}_\sigma(Z_1,\ldots,Z_q).
\end{equation*}

It now follows by Möbius inversion that $\widetilde{\kappa}_q = \kappa_q$, so every mixed cumulant involving $Z$ and $P_{(3,+)}$ vanishes. Invoking Theorem~\ref{thm:Speicher-cumulants-detect-freeness}, which tests freeness via cumulants on the set $S$, we see that the subalgebra they span is free in $(A_G^3,\tau_G)$.
\end{proof}

The next lemma records the loop-counting discussed above.

\begin{lemma}\label{lem:floating-loop}
For $q \geq 1$ and $Z_1,\ldots,Z_q\in \{Z\}\cup P_{(3,+)}$ let $D$ and $E$ be as in the proof of 
 Proposition~\ref{prop:C-and-P3-free}. Let $\pi \in NC(D)$. Then the $\pi$-term diagram as in Figure \ref{fig:pi_term} 
 contains exactly $2\bigl(|D|-|\pi\sqcup\widetilde{\pi}|+1\bigr)$ floating loops.
\end{lemma}

\begin{proof}
This is a modification of \cite[Proposition~6.6]{JayakumarKodiyalam2025}. In the original construction, the labeled $(2,+)$-boxes were kept as labeled boxes, so floating loops contributed only from the lower half of the diagram. Here, we use explicit identity tangles instead of labeled upper boxes. This means floating loops now contribute from both the upper and lower parts of the diagram, yielding the factor of 2 in the loop count formula.
\end{proof}

Combining Proposition~\ref{prop:AG3-gen} with Proposition~\ref{prop:C-and-P3-free} yields the following.

\begin{proposition}\label{prop:AG3-free-product}
The algebra $A_G^3$ is the algebraic free product $\mathcal{C} * P_{(3,+)}$. Hence,
\[
M^3 = {vN}(A_G^3) = {vN}(Z) * P_{(3,+)}.
\]
\end{proposition}

\begin{proof}
By Proposition~\ref{prop:AG3-gen}, 
the subalgebras $\mathcal{C}$ and $P_{(3,+)}$ generate $A_G^3$. 
Proposition~\ref{prop:C-and-P3-free} 
shows that they are free with respect 
to $\tau_G$, so $A_G^3$ satisfies the 
universal property of the algebraic 
free product $\mathcal{C} * P_{(3,+)}$. Taking double commutants gives the described decomposition of $M^3$.
\end{proof}

\subsection{Proof of the main theorem}

\begin{proof}[Proof of Theorem~\ref{thm:main-M3}]
Lemma~\ref{lem:Z-free-poisson} gives
\[
{vN}(Z) \cong \underset{1-\frac{1}{n}}{\mathbb{C}} \oplus \underset{\frac{1}{n}}{L\mathbb{Z}}.
\]
We know that $P_{(3,+)} \cong M_n(\mathbb{C})$ endowed with its normalized trace. Applying Proposition~\ref{prop:AG3-free-product}, we obtain
\[
M^3 = {vN}(A_G^3) = {vN}(Z) * P_{(3,+)} \cong \left(\underset{1-\frac{1}{n}}{\mathbb{C}} \oplus \underset{\frac{1}{n}}{L\mathbb{Z}}\right) * M_n(\mathbb{C}).
\]
By the free product formula of Equation~\eqref{eq:LF-matrix}, 
\[
\left(\underset{1-\frac{1}{n}}{\mathbb{C}} \oplus \underset{\frac{1}{n}}{L\mathbb{Z}}\right) * M_n(\mathbb{C}) \cong \LF{1+\frac{2(n-1)}{n^2}}.
\]
Thus $M^3$ is the interpolated free group factor $\LF{1+\frac{2(n-1)}{n^2}}$, as claimed.
\end{proof}


The diagrammatic route we follow pinpoints $M^3$ as an interpolated free group factor with no auxiliary factoriality argument required. Further, $M^3$ remembers only the order $n=|\mathcal{G}|$ of the group and forgets its isomorphism type entirely, as expected.

\section*{Acknowledgements}

The author warmly thanks Professor Jaya Iyer for her sustained encouragement and generous financial support, which created the space for this project to develop. Deep gratitude is also owed to Professor Vijay Kodiyalam for many insightful conversations and his patient guidance throughout the course of this work.

\end{document}